\documentclass{article}
\RequirePackage{amsthm,amsmath,amsfonts,amssymb,color}
\RequirePackage[numbers]{natbib}
\newcommand{\corO}{}
\def\P{\mathbb{P}}
\def\E{\mathbb{E}}
\newtheorem{theorem}{Theorem}

\newtheorem{corollary}[theorem]{Corollary}
\newtheorem{lemma}[theorem]{Lemma}

\theoremstyle{definition}
\newtheorem{remark}[theorem]{Remark}
\newtheorem{dfn}[theorem]{Definition}
\begin{document}

%\begin{frontmatter}
\title{Exponential growth of random infinite Fibonacci sequences}
%\title{A sample article title with some additional note\thanksref{T1}}
%\runtitle{Random infinite Fibonacci sequences}
%\thankstext{T1}{A sample of additional note to the title.}
\author{Ilya Goldsheid\footnote{School of Mathematics, Queen Mary University of London,
Mile End Road,London E1 4NS, England. email: i.goldsheid@qmul.ac.uk } 
\and Ofer Zeitouni\footnote{Department of Mathematics, Weizmann Institute of Science, POB 26, Rehovot 76100, Israel. email: ofer.zeitouni@weizmann.ac.il}}
\date{August 3, 2026}
\maketitle
\begin{abstract}
We consider the recursion  $X_{n+1}=\sum_{i=0}^n \epsilon_{n,i}X_{n-i}$, where $\epsilon_{n,i}$ are i.i.d. (Bernoulli) random variables taking values in $\{-1,1\}$, and $X_0=1$, $X_{-j}=0$ for $j>0$. We prove that
almost surely, $n^{-1}\log |X_n|\to \bar \gamma>0$, where $\bar \gamma$ is an appropriate Lyapunov exponent. This answers a question of Viswanath and Trefethen (\textit{SIAM J. Matrix Anal. Appl. 19:564--581, 1998}).
\end{abstract}
\iffalse


\end{frontmatter}
%%%%%%%%%%%%%%%%%%%%%%%%%%%%%%%%%%%%%%%%%%%%%%
%% Please use \tableofcontents for articles %%
%% with 50 pages and more                   %%
%%%%%%%%%%%%%%%%%%%%%%%%%%%%%%%%%%%%%%%%%%%%%%
%\tableofcontents
\fi
%%%%%%%%%%%%%%%%%%%%%%%%%%%%%%%%%%%%%%%%%%%%%%
%%%% Main text entry area:

\section{Introduction}
Let $a_{i,n}$ denote a triangular array of i.i.d.,
zero mean random variables of law $\mu$.
In their study of the condition number of random Gaussian matrices,
Viswanath and Trefethen \cite{VT} considered the recursion
\begin{equation}
  \label{eq-030524c}
\corO{  t_0=1,\ t_{n+1}=\sum_{i=0}^n a_{i,n} t_{n-i}/a_{n,n}}
\end{equation}
for the case when $\mu$ is the standard Gaussian law.
Using remarkable
explicit computations, they where able to compute
$\lim n^{-1} \log (\corO{\sum_{i=0}^n t_i^2})$
and prove that it converges almost surely as $n\to \infty$ to $\log 4$; they
also showed that this coincides with the exponential rate of growth
of the above-mentioned condition number.

It is natural  to ask similar questions for other distributions, and in fact
this question already appears in \cite{VT}, see also \cite{ET}.
A particularly interesting case is when
$\mu$ is the symmetric Bernoulli law  on $\{-1,1\}$.
In that case, the recursion coincides in law with the recursion
\begin{equation}
  \label{eq-1}
  X_{n+1}=\sum_{i=0}^n \epsilon_{n,i}X_{n-i},
\end{equation}
where $\epsilon_{n,i}$ are iid, zero mean, Bernoulli random variables
with values in $\{-1,1\}$, for which the explicit computation carried out in
\cite{VT} does not apply.
Partially motivated by this question,
Viswanath \cite{V} considered the case of a
random Fibonacci sequence, i.e. when
\eqref{eq-1} is replaced by
\begin{equation}
  \label{eq-1aa}
  F_{n+1}=\epsilon_{n,0} F_n+\epsilon_{n,1} F_{n-1}.
\end{equation}
In this case, the vector $(F_{n+1},F_{n})$ can be presented as a product of $2\times 2$ random matrices applied to $(F_{1},F_{0})$.
Using Furstenberg's theory, Viswanath proved that $|F_n|$ grows exponentially. He also
evaluated the rate of growth to arbitrary precision.

\corO{The main  goal} of this paper is to return to the Viswanath-Trefethen
question in the case of Bernoulli variables, and prove
an almost sure exponential rate of growth. That is,
%\section{Main result}
we consider the recursion \eqref{eq-1},
where $\epsilon_{n,i}$ are iid, zero mean,
%unit variance. The main case
%of interest will be when the $\epsilon_{n,i}$ are
Bernoulli random variables taking values in $\{-1,1\}$, and $\hat X_0=e_0=(1,0,\ldots)\in \ell_2$.
It will be convenient to
introduce the vector
\begin{equation}\label{eqX_n}
  \hat X_n=(X_n,X_{n-1},\ldots,X_0,0,\ldots)\in \ell_2.
\end{equation}
Our main result is the following
\begin{theorem}
  \label{thm-1}
  There exists a deterministic constant $\gamma>0$ so that
  \begin{equation}
    \label{eq-250424a}
    \lim_{n\to\infty} \frac{1}{n}\log |X_n|
    =\lim_{n\to\infty} \frac{1}{n}\log \|\hat X_n\|_2
    =\gamma>0, \quad a.s.
  \end{equation}
\end{theorem}
\begin{remark} We do not have an explicit expression for the constant $\gamma $.
%We do not know how to compute the constant $\gamma$. 
Numerical computations reported in \cite[Page 2483]{ET} suggest the value
$\gamma\sim \log(1.3272)\sim 0.12294$.
\end{remark}
\begin{remark}
  Our methods, which rely heavily on a result concerning products of
  random operators due to Ruelle and to Goldsheid-Margulis,
  can be extended to other laws $\mu$. We shall prove that the second limit in \eqref{eq-250424a}
  exists for a wide class of distributions $\mu$ and initial conditions $\hat X_0$.
  \corO{Our proof that the limit is deterministic, and the proof of the convergence of the first limit,} is specific to the Bernoulli case, but probably could be extended
  beyond that by using modern analogues of the  S\'{a}rk\H{o}zy-Szemer\'{e}di theorem we employ.
  It turns out that  for the recursion \eqref{eq-1} with the law $\mu$
  being Gaussian, a different proof based on a certain contraction
  property can be given. We provide a sketch in Appendix \ref{app-B}.
\end{remark}
\subsection{Notation and conventions}

 We always suppose that the random variables we consider are defined on
a probability space $(\Omega,\mathcal{F},\P)$, $\E(\cdot)$ denotes the expectation
with respect to the probability measure $\P$, and $\E(\cdot\mid \cdot)$ is the conditional expectation.

Throughout the paper, all vectors are column vectors but we write $Y = (y_0,y_1,...)$
rather than $Y = (y_0,y_1,...)^T$.
Accordingly, we write $AY$, where $A$ is a matrix rather than $AY^T$.

The norm of $Y\in\ell_2$ is often written as $\|Y\|$ rather than $\|Y\|_2$. But we use the latter when we want to
emphasize the importance of the fact that $Y$ is considered as an element of $\ell_2$.

Many of our results don't require the distributions of $\epsilon_{n,i}$'s to be Bernoulli. We state here conditions \eqref{cond1} and
\eqref{cond2} for future references.
\begin{equation}\label{cond1}
 \text{The random variables $\epsilon_{n,i},\ i\ge0,$ are iid with $\E(\epsilon_{n,i})=0$ and $\E(\epsilon_{n,i}^2)=1$.}
\end{equation}
\begin{equation}\label{cond2}
 \E(\epsilon_{n,i}^4)<\infty.
\end{equation}
It will always be clear from the context whether we are dealing with the Bernoulli distribution
or with the more general case.

\subsection{A bird's eye view of the proof}
\label{sec-bird}
\corO{The claim \eqref{eq-250424a} contains in fact several distinct statements: that the second limit involving the $\ell_2$ norm of $\hat X_n$ exists and is deterministic a.s, that it is positive, and that it equals the first limit.  While related, the proofs of these claims employ quite different tools. We review next the different arguments involved.}

\corO{As with Viswanath's argument  \cite{V}, the first step is to look at the recursion \eqref{eqX_n} satisfied by $\hat X_n$, and observe that it corresponds to a product of i.i.d.  random 
operators $A_i$, defined in \eqref{eq-A},  acting on $\ell_2$ (generalizing the case of $2\times 2$ matrices in \cite{V}). That is, we can write $\hat X_n= A_n\cdots A_1 e_0=: S_n e_0$, where $e_0=(1,0,\ldots)$ is a unit vector in $\ell_2$.
It is relatively straightforward (using Kingman's sub-additive ergodic  theorem)  to prove that the  limit of $n^{-1} \log \|S_n\|_{c,2}$ exists and a.s. equals  a (positive) deterministic constant $\gamma$, referred to as the top 
Lyapunov exponent associated with the sequence $S_n$ (here, $\|S_n\|_{c,2}$ denotes the operator norm of $S_n$, in the space $H_{c,2}$
 introduced in Section \ref{sec-2.1.1}; we note that the operator norm of $S_n$ in $\ell_2$ is infinite!). Thus, the whole difficulty is in proving that the same limit occurs with a specific  deterministic initial condition, when one considers $n^{-1} \log \|\hat X_n\|_2$, and moreover that the same limit is obtained when considering $n^{-1} \log |X_n|$, that is, just the first coordinate of $\hat X_n$.}
  
\corO{Because we are dealing with an infinite dimensional space, Furstenberg's theory does not apply verbatim. Instead, a theory applicable 
 to Hilbert spaces,  which can be viewed as an extension of Oseledets' theorem, was developed by Ruelle \cite{Ru}. We will use a version
of this theory,
%However, we can replace it by an extension, going back to Ruelle \cite{R},
 due to Goldsheid and Margulis \cite{GM}.  The latter applies in particular to i.i.d. operators on a Hilbert space, of the form $U+K_n$  where $U$ is an isometry and $K_n$ is a  compact operator.
Our operators $A_n$ are not quite of this form, and so our first step is to move the discussion to another Hilbert space, denoted $H_{c,2}$, which is well suited for the application of the Goldsheid-Margulis theory.
This is explained in detail in Section \ref{sec-GM}. The discussion here is quite general and applies to sequences of random variables satisfying \eqref{cond1} and \eqref{cond2}. The conclusion of the Goldsheid-Margulis theory is then that $H_{c,2}$ can be decomposed as a direct sum of a (random) finite dimensional subspace $H_0$  and  a random subspace $\mathcal{H}$, such that 
if $Y\notin \mathcal{H}$ then the exponential growth rate of $\|S_n Y\|_{c,2}$ coincides a.s. with $\gamma$, and more generally for any $Y\in H_{c,2}$, the limit $n^{-1} \log \|S_n Y\|_{c,2}$ exists a.s. Here, $\|\cdot\|_{c,2}$ denotes the norm in $H_{c,2}$.}

\corO{Once the fact that $n^{-1} \log \|S_n Y\|_{c,2}$ converges (to a possibly random limit) is established, elementary arguments (contained in Section \ref{sec-positivity}) show that the same limit 
is obtained for the $\ell_2$ norm, and further that the limit is a.s. positive for deterministic $Y$. Again, these facts are proved under the assumptions \eqref{cond1} and \eqref{cond2}.}

\corO{The crucial part of the proof of the second equality in \eqref{eq-250424a} consists in showing that the limits discussed so far are in fact a.s. deterministic, for  initial condition $e_0\in \ell_2$. This is 
equivalent to showing that $e_0\not\in \mathcal{H}$, and here we use specific properties of Bernoulli variables, and in particular the anti-concentration of weighted sums of i.i.d. Bernoulli random variables, as provided by the
Littlewood-Offord theorem, see Lemma \ref{Lemma10} and its proof. A slightly more sophisticated anti-concentration result (the S\'{a}rkozy-Szemer\'{e}di theorem)
 is then used to prove the first equality in \eqref{eq-250424a}. All this is contained in Section \ref{sec-proofthm1}.}

\section{Proofs}
\corO{We now turn to the program outlied in Section \ref{sec-bird}.}
\subsection{Reduction to a question about products  of operators and  the Goldsheid-Margulis theory}
%a result from \cite{GM}}
\label{sec-GM}
\subsubsection{Setup and the space $H_{c,2}$} \label{sec-2.1.1}
If, as above, $\hat X_0=(1,0,0,...)$ then we can write
\begin{equation}
  \label{eq-1a}
  \hat X_{n}=A_n\cdots A_1 \hat X_0
\end{equation}
where
\begin{equation}\label{eq-A}
  A_n=\left(\begin{array}{llll}
    \epsilon_{n,0}&\epsilon_{n,1}& \ldots&\ldots\\
    1&0&\ldots&\ldots\\
    0&1&\ldots&\ldots\\
    \ldots&0&1&0\\
    \vdots&\vdots&\vdots&\vdots
  \end{array}
  \right).
\end{equation}
The vectors in \eqref{eq-1a} are exactly the ones defined by \eqref{eqX_n}.

The sequence $\hat X_n$ is well defined because these vectors have finite support.
But if the initial condition is an arbitrary vector $Y_0\in\ell_2$, then one has to be more careful
because the matrices $A_n$ viewed as operators acting on $\ell_2$ have almost surely infinite norms.
Nevertheless, by the Khinchin-Kolmogorov theorem,  condition \eqref{cond1} implies that
the series $\sum_{i=0}^{\infty}\epsilon_{n,i}y_i$ converges with probability 1 if $\sum_{i=0}^{\infty} y_i^2 <\infty$. Therefore,
for every $Y_0\in \ell_2$ the sequence
\begin{equation}\label{eqY_n}
  Y_n=A_{n}...A_1Y_0
\end{equation}
is well defined with probability 1. However, in order to control the behaviour of the sequence
$Y_n$ we are going to use theorems that require the $A_n$'s to be bounded operators.
To overcome this dilemma, we introduce a family of Hilbert spaces.
%$A_n$ thus have to be careful with the definition of the domain of $A_n$.
Namely, for $c\geq 0$ real, set
$$H_{c,2}=\{x\in \ell_2: \sum_{i=0}^\infty e^{ci}x_i^2<\infty\},$$
and denote the natural norm in $H_{c,2}$ by $\|\cdot\|_{c,2}$ (obviously, $H_{0,2}=\ell_2$.)
Then, if \eqref{cond1} is satisfied and $c>0$, $A_n$ is almost surely a bounded operator from $H_{c,2}$ to itself.

\subsubsection{The Goldsheid-Margulis theory}
\corO{We shall now state a version of a result from \cite{GM} that we are going to use (see \cite{Ru} for a similar statement).}
Let $U$ be an isometry  and let $K_n,\ n\ge1$ be a sequence of iid compact random operators acting
on a Hilbert space $H$. Set $V_n=(U+K_n)(U+K_{n-1})...(U+K_{1})$. By the  Kingman sub-additive ergodic theorem, the following limit exists almost surely:
%while its positivity is a consequence of Lemma \ref{lemma2} below. Also suppose that almost surely
\begin{equation}\label{normV}
  \lim_{n\to\infty}\frac{1}{n}\log\| V_n\|_H=:\bar\gamma.
\end{equation}
(We show below in Lemma~\ref{lemma2} that in our context $\bar \gamma >0$.)
%(We show below in Lemma \ref{lemma2} that in our context as described below,  $\bar \gamma>0$.)
\begin{theorem}[Goldsheid-Margulis {\cite[Theorem 1.9]{GM}}]
\label{theo-GM}
The sequence of products $V_n$ has
the following properties which are satisfied almost surely: there is a (random) decomposition $H=H_0\oplus\mathcal{H}$ such that:
 \begin{itemize}
   \item[(a)] $H_0$ is finite dimensional.
   %\item $\lim_{n\to\infty} \frac1n \log \|A_n\cdots A_0\|_H =\gamma>0$, a.s.
   \item[(b)] For $v\not\in \mathcal{H}$,
    $\lim_{n\to\infty} \frac1n \log \|V_nv\|_H =\bar\gamma$, a.s.
  \item[(c)] There exists $\bar\gamma'<\bar\gamma$ so that for any $v\in \mathcal{H}$,
    $\lim_{n\to\infty} \frac1n \log \|V_n v\|_H\leq \bar\gamma'$.
 \end{itemize}
\end{theorem}
\begin{remark} 
\corO{The statement in [Theorem 1.9]{GM} is given for $U$ which is a unitary operator. However,
 a careful examination of
the proof  reveals that it works if $U$ is just an isometry (rather than unitary) operator.}
\end{remark}

\subsubsection{Application of Theorem \ref{theo-GM} to our setup}
\label{sec-setupapp}
\corO{We assume \eqref{cond1} and \eqref{cond2}, and  take  $H=H_{c,2}$.  The matrices $A_n$ defined by \eqref{eq-A} have the form
$A_n=U+K_n$, where $U$ is the right shift operator (which is not an isometry on $H$),} and $K_n$ is an a.s.  bounded operator whose range
is the one-dimensional subspace
of $H_{c,2}$ generated by the vector $e_0=(1,0,\ldots)$.
More precisely, if $Y=(y_0,y_1\ldots)\in H_{c,2}$ then
\[
UY=(0,y_0,y_1,\ldots) \text{ and } K_nY=\left(\sum_{i=0}^{\infty} \epsilon_{n,i} y_i\right)e_0.
\]
Since
\begin{equation}
  \label{eq-bernoulli1}
  |\sum \epsilon_{n,i} y_i|\leq \left(\sum \epsilon_{n,i}^2 e^{-ci}\right)^{1/2} \left(\sum y_i^2 e^{ci}\right)^{1/2}
  = C\|Y\|_{c,2},
\end{equation}
where $C=C(n,\omega)$ is a random constant, we see that $K_n$ is indeed \corO{a.s.} a bounded rank 1 operator and hence it also is a compact operator.

Next, note that the operator norm $\|U\|_{c,2}=e^{c/2}$, and that $e^{-c/2} U$
is an isometry operator on $H_{c,2}$.
It follows that $e^{-c/2} A_n$ is the sum of a deterministic isometry operator
and a random (bounded) compact operator of rank 1.
Hence the properties (a), (b), (c) \corO{of Theorem \ref{theo-GM}} hold true for products of matrices $A_n$
acting on $H_{c,2}$.
% Note also that $A_n^{-1}A_n =I$.

\subsection{Proof of existence and positivity of  the second limit in \eqref{eq-250424a}}%\phantom{}
\label{sec-positivity}
\corO{In this section we will first show the positivity  of the liminf (instead of limit)  in the second term in \eqref{eq-250424a}. This will be done for arbitrary initial condition in $\ell_2$.
Since the existence of the limit  is provided by Theorem \ref{theo-GM}  only with the $H_{c,2}$ norm, in the second part of the section we show 
the existence of the \corO{possibly random} second limit in \eqref{eq-250424a}, \corO{under \eqref{cond1} and \eqref{cond2},}  for any initial condition $Y_0=\hat X_0$
from a certain subset of the unit sphere. }

%, and $\hat X_0$ belonging to this subset.
\corO{We will first show the positivity (with the $\ell_2$ norm) of the liminf  in the second term in \eqref{eq-250424a}.
We need the following elementary lemma.}
\begin{lemma}\label{lemma1} Let $A$ be a matrix that has the same distribution as $A_1$.
We denote the entries of its first row by $\epsilon_i,\ i\ge 0$, where $\epsilon_i$ are iid random variables
with $\E(\epsilon_i)=0$, $\E(\epsilon_i^2)=\sigma^2$, $\E(\epsilon_i^4)=D<\infty$. Then there is $\alpha <1$
such that
\begin{equation}\label{alpha}
\corO{\left(\E\left(\|AY\|_{2}^{-1}\right)\right)^{2}\leq } \E\left(\corO{\|AY\|_{2}^{-2}}\right)\le \alpha^2\ \text{ for all  fixed $Y\in \ell_{2}$ with $\|Y\|_{2}=1$}.
\end{equation}
%for any $Y\in \ell_{2}$
%with $\|Y\|_{2}=1$ and any $\kappa,\ 0<\kappa<\sigma^2$ the following inequality holds true:
%\[
%\E\left(\|AY\|_{2}^{-1}\right)\le 1-\frac{\kappa\varphi(\kappa))}{1+\kappa}, \text{ where } \varphi(\kappa)=\frac{(\sigma^2-\kappa)^2}{7D}.
%\]
\end{lemma}
\begin{remark}
It is important that $\alpha$ depends only on $\sigma^2$ and $D$ and the estimate \eqref{alpha} is uniform with respect to $Y$, as long as the latter does not depend on $A$.
%The proof of this lemma will be given in the Appendix.
\end{remark}
\begin{proof}[Proof of Lemma \ref{lemma1}]
\corO{If $\|Y\|=1$ then $\|AY\|^2=1+\zeta$,} where $\zeta=\left(\sum_{i=0}^{\infty}\epsilon_iy_i\right)^2$.
Note that $\E(\zeta)=\sigma^2$ and $\E(\zeta^2)\le 7D$. \corO{Recall the Paley--Zygmund inequality \cite[Lemma 19]{PZ}:
% that states that 
for a non-negative 
random variable $W$ and $\theta\in [0,1]$, 
$ \P(W \geq \theta \E W)\geq (1-\theta)^2 \frac{(\E W)^2}{\E W^2}. $
With $W=\zeta$ and $a=\theta \E\zeta$ we thus obtain that
% By Lemma \ref{zeta>a} with $h=1$, 
for any $ 0<a<\sigma^2$,
we have 
$$p=\P(\zeta\ge a)\ge \frac{(\sigma^2-a)^2}{7D}.$$}
\corO{Using that
$\zeta\ge aI_{\zeta\ge a}$ we arrive at}
\[
 \corO{\E(\|AY\|^{-2}) }=\E\left(\frac{1}{1+\zeta}\right)\le
\E\left( \frac{1}{1+aI_{\zeta\ge a}}\right)
= 1-p +\frac{p}{1+a}=1-\frac{pa}{1+a}<1.\]
We can now set \corO{$\alpha^2=1-\max_a\frac{a(\sigma^2-a)^2}{7D(1+a)}$.}
\end{proof}
%The proof of Lemma \ref{lemma1} is deferred to Appendix \ref{app-A}.
Next, set $S_n=A_n...A_1$. We have the following lemma.
\begin{lemma}\label{lemma2}
\corO{Assume \eqref{cond1} and \eqref{cond2}.}
For every fixed  $Y_0\in \ell_{2},\ \|Y_0\|=1,$
$$\liminf_{n\to \infty}\frac1n\log\|S_n Y_0\|>0 \ \text{ almost surely.}$$
\end{lemma}
\begin{proof} 
  %Set $S_n=A_n...A_1$ and 
  For a given $Y_0$ define the vectors $Z_n=S_nY_0/\|S_nY_0\|$. Note that
\begin{equation}\label{eq-L6}
 \|S_n Y_0\|=\|A_{n}Z_{n-1}\|\cdot\|S_{n-1}Y_{0}\|=\|A_{n}Z_{n-1}\|\cdot\|A_{n-1}Z_{n-2}\|\cdot...\cdot\|A_{1}Y_{0}\|.
\end{equation}
Next, by the Markov inequality, for any $\delta>0$
\begin{equation}\label{eqnLE>0}
\P\left(\frac1n\log \|S_n Y_0\| < \delta\right) =\P\left(\|S_n Y_0\|^{-1}>e^{-n\delta}\right)\le
e^{n\delta}\E\left(\|S_n Y_0\|^{-1}\right).
\end{equation}
%Since $\|A_{n}Y_{n-1}\|=(1+g_n^2)^{1/2}$, where $g_n=\sum_{i=0}^\infty \epsilon_{n,i} y_{n,i}$
We now use \eqref{eq-L6} and compute the expectation in the rhs of \eqref{eqnLE>0} by conditioning on $S_{n-1}Y_{0}$:
\begin{align*}
\E\left(\|S_n Y_0\|^{-1}\right) & =
\E\left( \E\left(\|A_{n}Z_{n-1}\|^{-1}\cdot\|S_{n-1}Y_{0}\|^{-1}\mid\,S_{n-1}Y_{0}\right)\right) \\
 & =\E\left(\|S_{n-1}Y_{0}\|^{-1}\E\left(\|A_{n}Y_{n-1}\|^{-1}\mid\,S_{n-1}Y_{0}\right)\right).
\end{align*}
By Lemma \ref{lemma1} the conditional expectation $\E\left(\|A_{n}Y_{n-1}\|^{-1}\,\mid \,S_{n-1}Y_{0}\right)\le \alpha$
and hence
\[
\E\left(\|S_n Y_0\|^{-1}\right)\le \alpha\E\left(\|S_{n-1}Y_{0}\|^{-1}\right)\le \alpha^n.
\]
We thus have
\begin{equation}\label{eqnLE>0.1}
\P\left(\frac1n\log \|S_n Y_0\| < \delta\right) \le e^{n\delta}\alpha^n,
\end{equation}
and we see that if $\delta<-\log\alpha$ then the rhs in \eqref{eqnLE>0.1} decays exponentially fast.
By the Borel-Cantelli lemma, with probability 1 the inequality $\frac1n\log \|S_n Y_0\| < \delta$ can be satisfied
only for finitely many $n$'s. The lemma is proved.
\end{proof}

\corO{We next turn to the proof of existence of the second limit in \eqref{eq-250424a}, which we do  by showing it is the same limit as when the norm is taken in $H_{c,2}$, whose existence follows from Theorem \ref{theo-GM}. 
 We will prove it for a certain class of initial conditions. To introduce these, we need the following. Let $B=\{(z_0,z_1,\ldots): \sum_{i=0}^\infty z_i^2=1\}$ denote the unit sphere in $\ell_2$.}
\begin{dfn}
\corO{$M_\alpha$ is the set of probability distributions $\nu$ on $B$ so that
\iffalse
the
unit sphere  in $\ell_2$  that have the following property.
If $\nu\in M_\alpha$ is the distribution of a random vector $Z=(z_0,z_1,z_2,...)\in\ell_2$ then
\fi 
$
\E_\nu(|z_i|)\le \alpha^i.
$}
\end{dfn}

\begin{lemma}
\label{cor-8}
Set $Z_n=S_nY/\|S_nY\|$, $Z_n=(z_{n,0}, z_{n,1}, z_{n,2},...)$, where $Y$ is distributed according to $\nu_0$ and is independent of $A_i,i\geq 1$ \corO{satisfying the assumptions of Lemma \ref{lemma1}.} Let $\nu_n$ be the distribution of $Z_n,\ n \ge 1$.
If $\nu_0\in M_\alpha$ then also $\nu_n\in M_\alpha$ for all $n\ge 1$, that is
\begin{equation}\label{Z_n}
  \E(|z_{n,i}|)\le\alpha^i.
\end{equation}
\end{lemma}
\begin{proof}
\corO{The lemma is proved by the following induction. Suppose that
\begin{enumerate}
  \item[(a)] $A$ is a matrix with properties listed in Lemma \ref{lemma1}.
  \item[(b)] $Y\in \ell_2$ is a unit random vector with distribution $\nu\in M_\alpha$,
where $\alpha$ is the same as in \eqref{alpha}.
  \item[(c)] $A$ and $Y$ are independent.
\end{enumerate}
Let $\nu_1$ denote the distribution of $Z=AY/\|AY\|$ . Then we claim that $\nu_1\in M_\alpha$.
(The claim in the lemma follows by taking $\nu$ to be the law of $Z_n$ and $\nu_1$ the law of $Z_{n+1}$, and iterating.)}

To see the claim concerning $\nu_1$, write $Z=(z_0,z_1,\ldots)$ and $Y=(y_0,y_1,\ldots)$.  Then $\E(|z_0|)\le 1$ because $\|Z\|_2=1$. Next, for $i\ge 1$ we have $z_i=y_{i-1}\|AY\|^{-1}$ and so
\[
 \E|z_i|=\E\left(|y_{i-1}|\cdot\|AY\|^{-1}\right)=\E\left(|y_{i-1}|\cdot\E\{\|AY\|^{-1}\mid Y\}\right)
 \le \alpha \E\left(|y_{i-1}|\right)\le \alpha^i,
\]
 where the estimate of the expectation of $\|AY\|^{-1}$ conditioned on $Y$ is due to Lemma \ref{lemma1} \corO{and the bound on the expectation of $|y_{i-1}|$ is because $Y$ is distributed according to $\nu\in M_\alpha$.}
 %$and the rest is  a straightforward induction in $i$.
\end{proof}
\begin{remark}
The distribution of $Y$ in Lemma \ref{cor-8} can be supported by a single vector. To relate the notation of Lemma \ref{cor-8} to our running convention, recall \eqref{eqX_n} and 
note that if $\nu_0=\delta_{e_0}$ then $Z_n \|\hat X_n\|=\hat X_n$ and $z_{n,i} \|\hat X_n\|=X_{n-i}$ for $i\leq n$.
\end{remark}
\corO{We finally show that limits for the norm in $H_{c,2}$ transfer to the $\ell_2$ norm, as long as the initial condition has law in $M_\alpha$. }
\begin{theorem}\label{Theorem9}
\corO{Assumer \eqref{cond1} and \eqref{cond2}.}
Suppose that $0<c<-\log\alpha$ and that a random vector $Y$ with distribution $\nu_0\in M_\alpha$
is independent of the sequence $(A_i)_{i\ge0}$. Then almost surely the following limits exist and are equal:
\begin{equation}\label{lim}
\lim_{n\to\infty}\frac{1}{n}\log \|S_n Y\|_{c,2}=\lim_{n\to\infty}\frac{1}{n}\log \|S_n Y\|_{2}.
\end{equation}
\end{theorem}
\begin{proof}
\corO{The existence of the first limit in \eqref{lim} follows from Theorem \ref{theo-GM}, as explained in Section \ref{sec-setupapp}.}
%\cite[Theorem 1.9]{GM}.
 It thus suffices to show that
\begin{equation}\label{lim2}
\lim_{n\to\infty}\frac{1}{n}\left(\log \|S_n Y\|_{c,2}-\log \|S_n Y\|_{2}\right)=
\lim_{n\to\infty}\frac{1}{n}\log \left(\|S_n Y/ \|S_n Y\|_{2}\|_{c,2}\right)=0.
\end{equation}
Note that \eqref{lim2}  holds if we show that
\begin{equation}\label{lim3}
  \lim_{n\to\infty}\frac{1}{n}\log \left(\|Z_n \|_{c,2}\right)=0,
\end{equation}
where the vector $Z_n=S_n Y/ \|S_n Y\|_{2}=(z_{n,0},z_{n,1},...)$ has distribution $\nu_n\in M_\alpha$ \corO{by Corollary \ref{cor-8}, and therefore
 $\E(z_{n,i}^2) \le \alpha^i$. But then}
% we get
\[
\E(\|Z_n \|_{c,2}^2)=\E\left(\sum_{i=0}^{\infty}e^{ci}z_{n,i}^2\right)\le \sum_{i=0}^{\infty}e^{ci}\alpha^i.
\]
The right hand side in the last display does not depend on $n$ and hence the limit in \eqref{lim3} is 0 almost surely.
\end{proof}

\subsection{Proof of Theorem \ref{thm-1}}
\label{sec-proofthm1}
\corO{As discussed in Section \ref{sec-bird},} the heart of the proof of Theorem \ref{thm-1} lies in the following lemma.
\begin{lemma}\label{Lemma10}
  We have that $e_0\not \in \mathcal{H}$, a.s.
\end{lemma}
\begin{proof}
  We begin by showing that $P(e_0\not \in \mathcal{H})>0$. Indeed, if $P(e_0\in \mathcal{H})=1$, then
also $P(e_j\in \mathcal{H})=1$ for any $j\ge 1$. Indeed, under the assumption, note that $A_0 e_0\in \mathcal{H}$. On the other hand, $A_0 e_0=e_1+\xi e_0$ for an approprite random variable $\xi$. Since $e_0\in \mathcal{H}$ almost surely, it follows that necessarily $e_1\in \mathcal{H}$. The claim for general $j$ follows by induction.

\iffalse
 \[\textrm{span} (e_0,\ldots,e_n)\subset
  \textrm{span} (e_0, A_{\epsilon_0} e_0, \ldots, A_{\epsilon_n}\cdots A_{\epsilon_0} e_0).\]
(An explicit choice is $\epsilon_{i}(k)=1$ for $k\leq i$.
Note that for a fixed $k$, this choice of coefficients has positive
probability under $P$.)
Since
  $\textrm{span} (e_0,\ldots,e_n)$ is dense in $H$, it follows that there exists a random $k$ so that
  $\textrm{span} (e_0,\ldots,e_k)\not \subset \mathcal{H}$.
  Considering now the matrices
  $B_{n+k}:=A_n\cdots A_0 A_{\epsilon_k} \cdots A_{\epsilon_0}$,
  we conclude that $\lim \frac{1}{n} \log \|B_n e_0\|_H=\gamma_c$.
  It follows that $P(e_0\not \in \mathcal{H})>0$.

  In the next step,
\fi
 We fix $\delta>0$ and show that
  $P(e_0\not \in \mathcal{H})>1-\delta$.
 For any $v=(v_0,v_1,\ldots)\in H$, set
  \[I_v=\min\{i: v_i\neq 0\}.\]
  Introduce the event
  \[\mathcal{A}_j=\{\exists w\in H_0: I_w\leq j\}\]
  Note that $a_j:=P(\mathcal{A}_j)\to 1$ as $j\to \infty$. Fix now $j_0$ so that  $a_{j_0}\geq 1-\delta/2$.

  Recall \corO{Erd\H{o}s' version of the Littlewood-Offord lemma \cite{Er}:} there is a universal constant $c$ such that  with $\epsilon_i$ iid standard Bernoulis and $b_j$ nonzero deterministic
  integers,
  \begin{equation}
\label{eq-Erdos}
\max_T \P(\sum_{i=1}^k \epsilon_i b_i=T)\leq c/\sqrt{k}.\end{equation}
  Choose now $k_0$ such that $c/\sqrt{k_0/2}<\delta/4$.

  Let $\widetilde{A}_i$ be i.i.d., independent of the $A_i$s and equidistributed as them. Set
$W_k=\widetilde A_k\cdots \widetilde A_0$. We first note that the entries of
$W_k e_0$ are all integers. By an application of the Littlewood-Offord
theorem, there exists $k_1$ so that with
\[\mathcal{B}_{k_1}=\cup_{k\geq k_1-1} \{|\{ i\leq k: (W_k e_0)_i\neq 0\}|<k/2\},\]
we have that $\P(\mathcal{B}_{k_1})\leq \delta/8$.

Fix $k=(k_0+k_1+j_0)$ and set $B_{n}=A_n\cdots A_0 \cdot W_k$. Note that
$\theta\!:=$ $\lim_{n\to \infty}$ $n^{-1}\!\log \!\|B_n e_0\|$ has the same law as $\lim_{n\to \infty} n^{-1} \log \|A_n\cdots A_0 e_0\|$. We will show that
$ \P(\theta<\gamma)<\delta$.

Assume that $\mathcal{A}_{j_0}$ holds. Fix  $w\in H_0$ (random) which achieves the event in $\mathcal{A}_{j_0}$. Let $j_1\leq j_0$ be the minimal  index $j$ with $w_j\neq 0$.
%achieving $\mathcal{A}_{j_0}$.
We have that $\{\theta<\gamma\}\subset \{\langle W_k e_0,w\rangle_H=0\}.$
We will show that $\P(
\langle W_k e_0,w\rangle_H=0)<\delta$.
In fact, we will show that
\[\P(
\langle W_k e_0,w\rangle_H=0| w, \mathcal A_{j_0})\leq \delta/2.\]
Indeed, consider
the event \[\mathcal{C}_{k_1}=\{|\{ i\leq k_1-1: (W_{k_1} e_0)_i\neq 0\}|>k_1/2\}.\]
(The event $\mathcal{C}_{k_1}$ ensures that at least $k_1/2$ of the first $k_1$ coordinates of $W_{k_1}e_0$ are non-zero.)
By our choices and the definition of $\mathcal{B}_{k_1}$, we have that $\P(\mathcal{C}_{k_1})\geq 1-\delta/8$. Note that $j_1$ is a measurable function of  $w$.
Conditioned on  $w$ and $\tilde A_0,\ldots,\tilde A_{k-1-j_1}$, we have that $(W_k e_0)_{j_1}=(W_{k-j_1} e_0)_0$, and
$\langle W_ke_0,w\rangle_H=\sum_{j=j_1}^\infty c^j w_j (W_k e_0)_j$. Recalling
that $(W_k e_0)_j$ are integers, and conditioned on the sigma algebra $\mathcal{G}$ generated  by $\tilde A_0,\ldots,\tilde A_{k-1-j_1}$ and $w$, there is  (since $w_{j_1}\neq 0$ and $w_j=0$ for $j<j_1$
and
$(W_k e_0)_j$ are $\mathcal{G}$-measurable   for $j>j_1$)
a unique random variable $L$, $\mathcal{G}$-measurable, with
\[\{0=\sum_{j=j_1}^\infty c^j w_j (W_k e_0)_j\}=\{(W_{k-j_1} e_0)_0=L\}.\]
(The variable $L$ can be written as $L=\!-\sum_{j=j_1+1}^\infty\! c^{j-j_1} w_j (W_k e_0)_j/w_{j_1}$.)
Now, $(W_{k-j_1} e_0)_0=\sum \epsilon_i b_i$ for some integer
 $\mathcal{G}$-measurable coefficients $b_i$, and i.i.d. Bernoullis
$\epsilon_i$ independent of $\mathcal{G}$.
On the event $\mathcal{C}_{k_1}$ (which is $\mathcal{G}$-measurable), at least
$k_1/2$ of the integer coefficients $b_i$ are nonzero. It follows
from our choice of $k_0$ and the Littlewood-Offord theorem that on the event
$\mathcal{A}_{j_0}\cap \mathcal{C}_{k_1}$,
\[\P(
(W_{k-j_1} e_0)_0=L|\mathcal{G})=\P(\sum_{i=1}^{k-j_1} \epsilon_i b_i=L|\mathcal{G})\leq \delta/4.\]
Altogether, and using that $\mathcal{B}_{k_1}\subset \mathcal{C}_{k_1}$, we conclude that
\[\P(\langle W_k e_0,w\rangle_H=0
 )\leq \delta/4+\P(\mathcal{A}_{j_0}^c)+\P(\mathcal{B}_{k_1}^c)\leq \delta/4+\delta/2+\delta/8<\delta.\]
\end{proof}
\iffalse

Further, we claim that $P( (W_k e_0)_0=0|W_i,i<k)\leq 1/2$, and then also $P( (W_k e_0)_0=0)\leq c/\sqrt{n}$ for some $c>0$
by the Littlewood-Offord lemma \cite{Er}.
Let $j_1$ denote the first entry of $w$ which is non-zero.
It follows that $P(\langle W_{k+j_1} e_0, w\rangle_H=0)\leq c/\sqrt{k}$.

Set now
  the sequence
  $B_{n,k}=A_n\cdots A_0 \cdot W_{k+j_1}$.
  Note
  that the sequence $(B_{n,k})_{n\geq 0}$ is equidistributed
  as the sequence $(A_n)_{n\geq 0}$. Now, if
  $\limsup_{n\to\infty} \frac1n \log \|B_{n,k} e_0\|_H<\gamma$ then necessarily
  $\langle W_{k+j_1}e_0,w\rangle_H=0$.
  Thus, $P(\limsup_{n\to\infty} \frac1n \log \|B_{n,k} e_0\|_H<\gamma)\leq c/\sqrt{k}$, and therefore
  $P(\limsup_{n\to\infty}\frac1n \log \|A_n\cdots A_0 e_0\|_H<\gamma)\leq c/\sqrt{k}$. Since $k$ is arbitrary, the lemma follows.
\end{proof}
\fi
\begin{lemma}
  \label{lemma11}
  Under the conditions of the theorem,
  \begin{equation}
    \label{eq-030524a}
    \lim \frac1n \log \|A_n \cdots A_0  \corO{e_0}\|_2=\bar{\gamma}>0.
  \end{equation}
\end{lemma}
\begin{proof} The existence of the limit follows from Theorem \ref{Theorem9},
its positivity follows from Lemma \ref{lemma2}, and the fact that it is equal to the top Lyapunov exponent
follows from Lemma \ref{Lemma10}. %$\Box$
 \end{proof}

It remains to control the behavior of $X_n$. The necessary estimate is contained in the following lemma.
\begin{lemma}
  \label{lem-3}
Under the conditions of the theorem, we have that
\begin{equation}
  \label{eq-top}
  \lim_{n\to\infty} \frac1n \log |X_n|=\bar{\gamma},  \quad a.s.
  \end{equation}
\end{lemma}
\begin{proof}
  Note that
  \[\limsup_{n\to\infty} \frac1n \log |X_n|\leq
  \limsup_{n\to\infty} \frac1n \log \|\hat X_n\|_2 =\bar{\gamma}, \quad a.s\]
  by Lemma \ref{lemma11}. It thus suffices to provide a complementary lower
  bound.

  Fix $\varepsilon>0$. From the convergence of $(\log \|\hat X_n\|_2)/n$ to a constant $\bar{\gamma}>0$, we  deduce that $\sum_{i=n-j\epsilon n+1}^{n-(j-1)\varepsilon n} X_i^2\geq e^{2\bar{\gamma} n(1-(j+1)\varepsilon)}$, for all $n$ large and $j=1,\ldots,1/\sqrt{\varepsilon}$. In particular, for each such $j$ there
  exists $i_j\in [n-j\varepsilon n+1,n-(j-1)\varepsilon n]$ with
  $|\corO{X_{i_j}}|\geq e^{\bar{\gamma} n(1-(j+1)\varepsilon)}/n$.
  \corO{By  Erd\H{o}s's version \eqref{eq-Erdos} of the Littlewood-Offord lemma  \cite{Er},}
  it follows that conditionally on $ X_i,i\leq n$, we have that
  $X_{n+1}\geq e^{\bar{\gamma} n(1-2\sqrt{\varepsilon})}$, with probability
  at least $1-c\sqrt{\varepsilon}$. Call such $X_n$ good.
  It follows that for all $n$ large, each
  block of size $\varepsilon n$ contains at least $\varepsilon n/2$
  such good indices, and a variant of this argument shows that at least
  $\varepsilon \sqrt{n}$ of them are distinct. \corO{The S\'{a}rkozy-Szemer\'{e}di theorem
  \cite{SS} (see also \cite{Hal}), which improves $\sqrt{k}$ to $k^{3/2}$ in \eqref{eq-Erdos} in case all coefficients $b_i$ are distinct,}
 then shows that
  for $\delta>0$,
  in fact, $n^{3/4-\delta}$ of them are distinct, with probability at least $1-
  e^{-cn}$. Another application of the S\'{a}rkozy-Szemer\'{e}di theorem
  \cite{SS} yields that
  $\P(|X_n|< e^{\bar{\gamma}
  n(1-2\sqrt{\varepsilon})})\leq  c_\varepsilon /n^{9/8-3\delta/2}$. The
  Borel-Cantelli lemma then shows that $|X_n|\geq e^{\bar{\gamma} n(1-2\sqrt{\varepsilon})}$, for all large $n$. Since $\varepsilon$ is arbitrary, we conclude that
  \[\liminf_{n\to\infty}  \frac1n \log |X_n|\geq \bar{\gamma}, \quad a.s.\]
This completes the proof.
\end{proof}
%\begin{acks}[Acknowledgments] 
{\bf Acknowledgements} O.Z. thanks Neil O'Connell for introducing
him to this problem decades ago, and many enlighting
discussions. He also
thanks Mark Rudelson for suggesting the approach leading to the proof of Lemma \ref{lem-3}, and Nick Trefethen for pointing out the numerical evaluation of $\gamma$ reported in \cite{ET}.  We thank the referees for their detailed reports  which led to an improved presentation of the results.
% The authors would like to thank ...
%\end{acks}
This work was supported by grant 615/24 from the Israel Science Foundation.
%\end{funding}

\appendix
\section*{The Gaussian case}
\label{app-B}
In this appendix we assume that the variables $\epsilon_{n,i}$ are i.i.d. and standard Gaussian.
 Introduce the vectors $Z_n=\hat{X}_n/\|\hat {X}_n\|_2$. We then have the
  recursion
  \begin{equation}
    \label{eq-2}
    Z_{n+1}=(g_n, Z_n)/(1+g_n^2)^{1/2},
  \quad g_n=\sum_{i=0}^\infty \epsilon_{n,i} Z_{n-i},
\end{equation}
where only finitely many terms do not vanish in the sum in \eqref{eq-2}. Note that \eqref{eq-2} shows that $\{Z_n\}$ is a Markov chain.
We always have $\|Z_n\|_2=1$. As in Lemma \ref{cor-8}, we have that the sequence $\{\|Z_n\|_{c,2}\}$ is tight if $c>0$ is small enough. It follows that the Markov chain
$Z_n$ possesses at least one invariant measure on $H_{c,2}$. Note that $\|\hat X_n\|_2=
\prod_{j=1}^n (1+g_j^2)^{1/2}$. This leads to the following:
%We summarize this in the following statement.
\begin{corollary}
  \label{cor-1}
 (i)  There exists an invariant measure $\mu_v$ (not necessarily unique)
 for
  the Markov chain defined by \eqref{eq-2}
  on $H_{c,2}$.\\
  (ii) Let $\mu_v$ denote an extremal invariant measure, and choose
  $X_0\sim \mu_v$. Then,
  \begin{equation}
    \label{eq-3}
    \lim_{n\to\infty} \frac{1}{n}
    \log \|\hat X_n\|=\lambda_v, \quad a.s.
  \end{equation}
  where
  \begin{equation}
    \label{eq-4}
  \lambda_v=\frac12
  \int \mu_v(dy) \E_\epsilon \log (1+\sum_{i=0}^\infty \epsilon_{n,i} y_i).
\end{equation}
\end{corollary}
\noindent Here,
$\E_\epsilon$ denotes expectation with respect to the i.i.d.
variables $\epsilon_{n,i}$ (and hence, the expression in the right hand side of \eqref{eq-4} does not depend on $n$).

We remark that $|X_{n+1}|$, conditioned on $\hat X_n$, is in the Gaussian case a centered Gaussian variable with variance equal to $\|\hat X_n\|_2^2$. Thus, a simple Borel-Cantelli argument shows that,
in the Gaussian case,
$|X_n|$ has the same exponential rate of growth as that of $\|\hat X_n\|_2$, and in the rest of this appendix we only discuss that.
 Toward evaluating the latter rate of growth, we are left with two tasks: showing that there is a unique invariant
measure $\mu_v$ in Corollary \ref{cor-1}, and proving that part (ii)
of the corollary remains true if $X_0=(1,0,\ldots)$. Both tasks
follow from the next theorem.
\begin{theorem}
  \label{theo-1gau}
  Let $Z_0,\tilde Z_0 \in H_{c,2}$, and let $Z_n,\tilde Z_n$ be the solutions
  to \eqref{eq-2} with the same i.i.d. standard Gaussian sequence $\{\epsilon_{n,i}\}$. Let
  $\rho_n=\langle Z_n,\tilde Z_n\rangle_2$. Then $\rho_n\to 1$, a.s.
\end{theorem}
\begin{proof}[Proof of Theorem \ref{theo-1gau}]
We introduce some notation. Let $a_n^2=1-\rho_n^2$. Let
$g_n,\tilde g_n$ be as in \eqref{eq-2}. We then have, after some algebra,
\begin{equation}
  \label{eq-5}
  a_{n+1}^2=\frac{a_n^2+(g_n-\tilde g_n)^2+2g_n \tilde g_n (1-\rho_n)}{(1+g_n^2)
  (1+\tilde g_n^2)}.
\end{equation}
In particular, we have, with $b_{n+1}=a_{n+1}^2/a_n^2$, and assuming
$\rho_n\geq 0$ (which we can always
assume, due to the invariance of the law of the
dynamics with respect to the transformation $Y_n\to -Y_n$),
\begin{equation}
  \label{eq-Bg}
  b_{n}=\frac{1+B_n^2+2g_n\tilde g_n /(1+\rho_n)}{(1+g_n^2)(1+\tilde g_n^2)},
\end{equation}
where $B_n=(g_n-\tilde g_n)/a_n$.

In the Gaussian case, there is a simplification: the law of $(g_n,\tilde g_n)$,
even when conditioned on $Y_n,\tilde Y_n$, is Gaussian, of zero mean
and covariance matrix $R_{\rho_n}$, where
$$R_\theta=\left(\begin{array}{ll}
  1&\theta\\
  \theta& 1
\end{array}
\right).$$ This allows us to represent $\tilde g_n=\rho_n g_n+a_n w_n$ where
$w_n$ is a standard Gaussian  independent of $g_n$. Set
\begin{eqnarray}
  \label{eq-7}
  F(\rho_n,g_n,w_n)&=& \log b_n\\
  &=&
  \log\left(1+
  \frac{ ((1-\rho_n)g_n+a_n w_n)^2}{a_n^2}+\frac{2g_n
  (\rho_n g_n+a_n w_n)}{(1+\rho_n)}\right)
  \nonumber \\
  &&\quad
-\log(1+g_n^2)-\log(1+(\rho_n g_n+a_n w_n)^2).\nonumber
\end{eqnarray}
A numerical evaluation shows that there exists a constant $\eta\sim -0.1395<0$
so that
\begin{equation}
  \label{eq-8}
  \max_{\rho_n\in [0,1]} \E_{g_n,w_n}F(\rho_n,g_n,w_n)=\eta<0.
\end{equation}
We thus have that $Q_n:=\sum_{i=1}^n \log b_i $
is a supermartingale. Further, $\E|b_i|<C$ for some universal constant
$C$, and thus $\E(e^{ |\log b_i|})$ is uniformly bounded.
It then
follows that $\limsup_{n\to \infty} n^{-1} \sum_{i=1}^n \log b_i \leq \eta<0$:
indeed, for $0<\theta<1/2$ and $\delta>0$, writing $\hat F(g_n,w_n)=
F(\rho,g_n,w_n)$,
\begin{eqnarray*}
  &&n^{-1}\log   \P(Q_n>(\delta+\eta) n) \leq -\theta (\delta+\eta)
  +\log \max_{\rho\in [0,1]} \E_{w_n,g_n}(e^{\theta \hat F(g_n,w_n)})
  \\
  &\leq & -\theta (\delta+\eta) +\log \max_{\rho\in [0,1]} \E_{w_n,g_n}(1+\theta
  \hat F(g_n,w_n)+\frac12 \theta^2 \hat F(g_n,w_n)e^{\theta \hat F(g_n,w_n)})\\
  &&\leq -\theta (\delta+\eta)+ \log (1+\theta \eta+ \theta^2\bar C)\leq -\theta \delta+\theta^2\hat C
\end{eqnarray*}
for some uniform constants $\bar C, \hat C$, where the next to last inequality
follows from Cauchy-Shwartz and $\theta<1/2$. Taking $\theta=q \delta$
with $q$ small enough,
it follows that $n^{-1}\log   \P(Q_n>(\delta+\eta) n)< 0$, as claimed.
  This proves the theorem.
\end{proof}

\iffalse
We claim more.
\begin{lemma}
  \label{lem-1}
  There exists $c>0$ such that the sequence $\{\|Y_n\|_{c,2}\}$ is tight.
\end{lemma}
\begin{proof}
  Write $Y_n=(Y_i^{(n)})_{i\geq 0}\in \ell_2$. We have, for $i<n$, that
  $$|Y_i^{(n+1)}|^2=\frac{G_{n-i}^2}{\prod_{k=0}^{n} (1+G_{n-k}^2)}$$
  where $G_l=\sum_{i=0}^l \epsilon_{l,i} Y^{(l)}_{l-i}$.
  Set $\eta_{k,n}=(1+G_{n-k}^2)^{-1}$ and $\mathcal{F}_{n}=\sigma(Y_j, j\leq n)$. We have that
  $$E (\log \eta_k|\mathcal{F}_{n-k})=- E(\log (1+G_{n-k}^2)|\mathcal{F}_{n-k})<\eta<0,$$
  for example using that $E(G_{n-k}^2|\mathcal{F}_{n-k})=1$ and
  $E(G_{n-k}^4|\mathcal{F}_{n-k})<C$ with $C$ uniform.
  (Note that $\eta$ is uniform in the choice of $\{Y_l\}$.) By Chebychev, we immediately conclude that for any $x>0$, and $i\leq n$,
  $$ P(\log |Y_i^{(n)}|^2> (\eta +x) i)\leq e^{-i I(x)}$$
  for some function $I$ that is strictly positive on $(0,\infty)$.
  Choose now $c>0$ so that $\alpha:=c +\eta<0$. We then have that
  \begin{eqnarray*}
    P(\sum_{i=1}^\infty e^{ci} |Y_i^{(n)}|^2>K)& \leq &
    \sum_{i=1}^\infty  P( e^{ci} |Y_i^{(n)}|^2>\frac{2K}{i^2})\\
    &\leq & \sum_{i=1}^\infty  P( \log |Y_i^{(n)}|^2>\log (\frac{2K}{i^2})-ci)\\
    &\leq & \sum_{i=1}^\infty  P( \log |Y_i^{(n)}|^2>\log (\frac{2K}{i^2})+
    (\eta -\alpha) i)\\
    &\leq& C(K) \sum_{i=1}^{\infty} e^{-i I(\alpha/2)},
  \end{eqnarray*}
  with $C(K)\to_{K\to\infty} 0$. The claim follows.
\end{proof}
We note that the proof of Lemma \ref{lem-1} shows that $C(K)$ in the last display of the proof decays
polynomially in $K$.
\fi

\bibliographystyle{plain}

\begin{thebibliography}{99}
\bibitem{ET} M. Embree and L. N. Trefethen, Growth and decay of random Fibonacci sequences, 
\textit{Proceedings: Mathematical, Physical and Engineering Sciences, London Mathematical Society} {\bf 455} (1999), 2471--2485.
  \bibitem{Er} P. Erd\H{o}s, On a lemma of Littlewood and Offord.
    \textit{Bull. Amer. Math. Soc.} {\bf 51} (1945), 898--902.
  \bibitem{GM} I. Ya. Goldsheid and G. A. Margulis,
    Lyapunov indices of a product of random matrices.
    \textit{Russ. Math. Surv.} {\bf 44} (1989), 11--71.
  \bibitem{Hal} G. Hal\'{a}sz,
    Estimates for the concentration function of combinatorial number theory and probability. \textit{Period. Math. Hungar.} {\bf 8} (1977), 197--211.
\bibitem{PZ} R.E.A.C. Paley and C. Zygmund,
On some series of functions. \textit{Math. Proc. of the  Cambridge Phylosophical  Society} {\bf 28} (1932), 190--205.
  \bibitem{Ru}
D. Ruelle,
Characteristic exponents and invariant manifolds in Hilbert space.
\textit{Ann. Math., Ser 2} {\bf 115} (1982), 243--290.
  \bibitem{SS} A. S\'{a}rk\H{o}zy and E. Szemer\'{e}di,
    \"{U}ber ein {Problem} von {Erd\H{o}s} und {Moser}.
    \textit{Acta Arith.} {\bf 11} (1965), 205--208.
  \bibitem{V} D. Viswanath, Random Fibonacci sequences and the number 1.13198824...,
    \textit{Mathematics of Computation} {\bf 69} (2000), 1131--1155.
  \bibitem{VT} D. Viswanath and L. N. Trefethen, Condition numbers of random triangular matrices, \textit{SIAM J. Matrix Anal Appl.} {\bf 19} (1998), 564--581.
\end{thebibliography}

\end{document}